\documentclass[
submission
, nomarks
]{dmtcs-episciences}

\usepackage[utf8]{inputenc}

\usepackage[round]{natbib}

\author{Alan Veliz-Cuba\affiliationmark{1}
  \and Lauren Geiser\affiliationmark{1}
}
\title[AND-OR Networks with Chain Topology]{The Number of Fixed Points of 
AND-OR Networks with Chain Topology}
\affiliation{
  University of Dayton, USA
}
\keywords{Boolean networks, fixed points, AND-OR networks, one-dimensional cellular automata}
\received{2016-XX-XX}

\usepackage{amsmath,amsfonts,amssymb,amsthm}
\usepackage{array,delarray}
\usepackage{color}

\newtheorem{theorem}{Theorem}[section]   
\newtheorem{corollary}[theorem]{Corollary}     
\newtheorem{lemma}[theorem]{Lemma}         
\newtheorem{proposition}[theorem]{Proposition}  

\theoremstyle{definition}

\theoremstyle{remark}
\newtheorem{example}[theorem]{Example}        

\numberwithin{equation}{section}     

\definecolor{darkgreen}{rgb}{0,0.6,0}

\begin{document}
\publicationdetails{VOL}{2016}{ISS}{NUM}{SUBM}
\maketitle
\begin{abstract}
AND-OR networks are Boolean networks where each coordinate function is either the AND or OR logical operator. We study the number of fixed points of these Boolean networks in the case that they have a wiring diagram with chain topology. We find closed formulas for subclasses of these networks and recursive formulas in the general case. Our results allow for an effective computation of the number of fixed points in the case that the topology of the Boolean network is an open chain (finite or infinite) or a closed chain. 
\end{abstract}

\section{Introduction}


Boolean networks, $f:\{0,1\}^n\rightarrow\{0,1\}^n$, have been used to study problems arising from areas such as mathematics, computer science, and biology \citep{Akutsu,AO,mendozamethod,CBN,Velizlacop}. A particular problem of interest is counting the number of fixed points ($x$ such that $f(x)=x$). To simplify this problem one can restrict the class of Boolean functions or the topology of the network \citep{AFK, ADG, Jarrah2007167, Aracena_FB_BRN,  Murrugarra201166,  SCBN,CBN,bnandnot,veliz2014piecewise, veliz2015dimension}, which in some cases allows to find effective algorithms or formulas in closed form.

In this manuscript we focus on the number of fixed points of AND-OR networks (each Boolean function is either the AND or the OR operator) that have open or closed chain topology. We first consider the case of finite open chain topology and find a recursive formula (Theorem \ref{thm:main}) and sharp lower and upper bounds. We then consider the case of infinite and closed chain topology, and show how they can be reduced to the case of finite open chain topology (Theorems \ref{thm:inf} \ref{thm:closed}). 

\section{Open Chain}
\label{sec:openchain}

Let $f=(f_1,\ldots,f_n):\{0,1\}^n\rightarrow \{0,1\}^n$ with $n\geq
2$ be an AND-OR network such that its wiring diagram is a
chain, Fig \ref{fig-chain}. That is, we consider Boolean networks of the form:
\begin{displaymath}
\begin{array}{lllllll}
f_1=x_2, & 
f_2=x_1\diamondsuit_2 x_3, &
f_3=x_2\diamondsuit_3 x_4, &
\ldots & , &
f_{n-1}=x_{n-2}\diamondsuit_{n-1} x_n, &
f_n=x_{n-1},
\end{array}
\end{displaymath}

where $\diamondsuit_i$ is the AND ($\wedge$) or the OR ($\vee$) operator.
\begin{figure}
\centerline{ \hbox{ \framebox{
\includegraphics[width=0.6\textwidth]{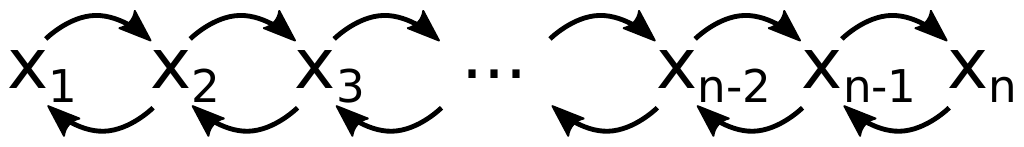}}}
 }\caption{Wiring diagram with open chain topology.}
 \label{fig-chain}
\end{figure}

Because this family of Boolean networks is completely determined by the sequence of logical operators $\diamondsuit_2,\diamondsuit_3,\ldots,\diamondsuit_{n-1}$, we can use this sequence to represent the network. Furthermore, consecutive occurrences of the same logical operator can be denoted as $\wedge^k$ or $\vee^k$.

We are interested in the number of fixed points of such Boolean networks. For simplicity we denote the elements of $\{0,1\}^n$ as binary strings (omitting parentheses). Also, we will use the notation $\textbf{0}=00\cdots 0$ and $\textbf{1}=11\cdots 1$, where the length of the strings will be clear from the context. Note that $\textbf{0}$ and $\textbf{1}$ are fixed points of all AND-OR networks with chain topology.

\begin{example}\label{eg:AOnet}
Our running example will be the AND-OR network 
\begin{displaymath}
\begin{array}{llllll}
f_1=x_2, & 
f_2=x_1\wedge x_3, &
f_3=x_2\wedge x_4, &
f_4=x_3\vee x_5, &
f_5=x_4\wedge x_6, &
f_6=x_5\vee x_7,\\
f_7=x_6\vee x_8, &
f_8=x_7\vee x_9, &
f_9=x_8\wedge x_{10}, &
f_{10}=x_9\wedge x_{11}, &
f_{11}=x_{10}\vee x_{12}, &
f_{12}=x_{11}.
\end{array}
\end{displaymath}
This network can be represented by the sequence of operators 
$\wedge\!\wedge\!\vee\!\wedge\!\vee\!\vee\!\vee\!\wedge\!\wedge\vee$. We can further simplify this representation to $\wedge^2\!\vee\!\wedge\!\vee^3\!\wedge^2\vee$. This AND-OR network has 13 fixed points listed in Table \ref{tab:AOnet} (first column).  
\end{example}

The next lemma states that the number of fixed points depends only on the powers of the operators. Since we do not know which operator is last ($\wedge$ or $\vee$), we will simply use ellipses without explicitly writing the last operator.

\begin{lemma}
The AND-OR networks $f=\wedge^{k_1}\vee^{k_2}\wedge^{k_3}\cdots$ and $g=\vee^{k_1}\wedge^{k_2}\vee^{k_3}\cdots$ have the same number of fixed points.
\end{lemma}
\begin{proof}
Consider $\phi:\{0,1\}^n\rightarrow \{0,1\}^n$ given by $\phi(x_1,\ldots,x_n)=(\neg x_1,\ldots,\neg x_n)$, where $\neg$ is the logical operator NOT. Using the fact that $\neg (p\wedge q)=\neg p \vee \neg q$ and $\neg (p\vee q)=\neg p \wedge \neg q$, it follows that $f(\phi(x))=\phi(g(x))$. Then, $x$ will be a fixed point of $g$ if and only if $\phi(x)$ is a fixed point of $f$. So, $\phi$ is a bijection between the fixed points of $g$ and $f$.
\end{proof}

Because we are interested in the number of fixed points, we will simply use $(k_1,k_2,\ldots,k_m)$ to refer to a network. For instance, the AND-OR network seen in Example \ref{eg:AOnet} can be represented simply by $(2,1,1,3,2,1)$. We denote the number of fixed points by $\mathcal{F}(k_1,k_2,\ldots,k_m)$. 

The following lemma states that consecutive variables that have the same logical operator must be equal.

\begin{lemma}\label{lemma:01}
Consider an AND-OR network $f$ represented by $(k_1,k_2,\ldots,k_m)$. Denote an element of the domain of $f$ by $\textbf{x}=(\textbf{x}^1,\textbf{x}^2,\ldots,\textbf{x}^m)$, where $\textbf{x}^1\in\{0,1\}^{k_1+1}$, $\textbf{x}^m\in\{0,1\}^{k_m+1}$, and $\textbf{x}^i\in\{0,1\}^{k_i}$ for $i=2,\ldots,m-1$. If $\textbf{x}$ is a fixed point of $f$, then $\textbf{x}^i=\textbf{0}$ or $\textbf{x}^i=\textbf{1}$ for $i=1,\ldots,m$.
\end{lemma}
\begin{proof}
Let $\textbf{x}$ be a fixed point of $f$. We use $(\textbf{x}^i)_j$ to denote the $j$-th coordinate of $\textbf{x}^i$. Note that $(\textbf{x}^1)_1=(\textbf{x}^1)_2$ and $(\textbf{x}^m)_{k_m}=(\textbf{x}^m)_{k_m+1}$ by definition of $f$ (the first and last coordinate functions of $f$ depend on single variables).

Now, the rest of the proof follows from the fact that if $q=p\wedge r$ and $r=q\wedge s$ or if $q=p\vee r$ and $r=q\vee s$, then $q=r$. This implies that consecutive variables, $(\textbf{x}^i)_j$ and $(\textbf{x}^i)_{j+1}$, that have the same logical operators must be the same.
\end{proof}

The next proposition states that the numbers $k_i$ in $\mathcal{F}(k_1,\ldots,k_m)$ can be assumed to be at most 2 for $2\leq i\leq m-1$, and 1 for $k_1$ and $k_m$. For example, this will imply that $\mathcal{F}(2,1,1,3, 2,1)=\mathcal{F}(1,1,1,2, 2,1)$ and $\mathcal{F}(2,5,3,1,4,3)=\mathcal{F}(1,2,2,1,2,1)$. 

\textit{Example} \ref{eg:AOnet} (cont.) We highlight the structure of the fixed points of 
$\wedge^2\!\vee\!\wedge\!\vee^3\!\wedge^2\vee$ in Table \ref{tab:AOnet} (second column).

\begin{proposition}\label{prop:reduction}
 $\mathcal{F}(k_1,k_2,\ldots,k_{m-1},k_m)=\mathcal{F}(1,\min\{k_2,2\},\ldots,\min\{k_{m-1},2\},1)$ for all positive integers $k_i$.
\end{proposition}
\begin{proof}
 We will use the notation of Lemma~\ref{lemma:01}.
 
We first show that $f=\wedge^{k_1}\vee^{k_2}\wedge^{k_3}\cdots$ and $g=\wedge\vee^{k_2}\wedge^{k_3}\cdots$ have the same number of fixed points. 
Let $\textbf{x}=(\textbf{x}^1,\ldots,\textbf{x}^m)$ be a fixed point of $f$. Then, by Lemma~\ref{lemma:01} we have  $\textbf{x}^1=\textbf{0}$ or $\textbf{x}^1=\textbf{1}$. Consider $\textbf{y}=(\textbf{z},\textbf{x}^2,\ldots,\textbf{x}^m)$, where $\textbf{z}=((\textbf{x}^1)_1,(\textbf{x}^1)_2)$. It can be checked that $\textbf{y}$ is a fixed point of $g$. Now, if $\textbf{y}=(\textbf{z},\textbf{x}^2,\ldots,\textbf{x}^m)$ is a fixed point of $g$, Lemma~\ref{lemma:01} implies that $\textbf{z}=\textbf{0}$ or $\textbf{z}=\textbf{1}$. We define $\textbf{x}=(\textbf{x}^1,\ldots,\textbf{x}^m)$ in the domain of $f$, where $\textbf{x}^1=\textbf{0}$ if  $\textbf{z}=\textbf{0}$ and $\textbf{x}^1=\textbf{1}$ if  $\textbf{z}=\textbf{1}$. Then, it can be checked that $\textbf{x}$ is a fixed point of $f$. This shows that $\mathcal{F}(k_1,k_2,\ldots,k_{m-1},k_m)=\mathcal{F}(1,k_2,\ldots,k_{m-1},k_m)$, and similarly it can be shown that $\mathcal{F}(1,k_2,\ldots,k_{m-1},k_m)=\mathcal{F}(1,k_2,\ldots,k_{m-1},1)$.

We now show that for $k_2\geq 2$, $f=\wedge^{k_1}\vee^{k_2}\wedge^{k_3}\cdots$ and $g=\wedge^{k_1}\vee^2\wedge^{k_3}\cdots$ have the same number of fixed points. The general case is analogous. Let $\textbf{x}=(\textbf{x}^1,\textbf{x}^2,\ldots,\textbf{x}^m)$ be a fixed point of $f$. Then, by Lemma~\ref{lemma:01} we have  $\textbf{x}^2=\textbf{0}$ or $\textbf{x}^2=\textbf{1}$. Consider $\textbf{y}=(\textbf{x}^1,\textbf{z},\textbf{x}^3,\ldots,\textbf{x}^m)$, where $\textbf{z}=((\textbf{x}^2)_1,(\textbf{x}^2)_2)$. It can be checked that $\textbf{y}$ is a fixed point of $g$. Now, if $\textbf{y}=(\textbf{x}^1,\textbf{z},\textbf{x}^3,\ldots,\textbf{x}^m)$ is a fixed point of $g$, Lemma~\ref{lemma:01} implies that $\textbf{z}=\textbf{0}$ or $\textbf{z}=\textbf{1}$. We define $\textbf{x}=(\textbf{x}^1,\textbf{x}^2,\ldots,\textbf{x}^m)$ in the domain of $f$, where $\textbf{x}^1=\textbf{0}$ if  $\textbf{z}=\textbf{0}$ and $\textbf{x}^1=\textbf{1}$ if  $\textbf{z}=\textbf{1}$. Then, it can be checked that $\textbf{x}$ is a fixed point of $f$. This shows that $\mathcal{F}(k_1,k_2,\ldots,k_{m-1},k_m)=\mathcal{F}(k_1,2,k_3,\ldots,k_{m-1},k_m)$ for $k_2\geq 2$.
\end{proof}

\textit{Example} \ref{eg:AOnet} (cont.) Proposition \ref{prop:reduction} guarantees that $\wedge^2\!\vee\!\wedge\!\vee^3\!\wedge^2\vee$ and $\wedge\!\vee\!\wedge\!\vee^2\!\wedge^2\vee$ have the same number of fixed points. We can consider the second AND-OR network as a ``reduced'' version of the original AND-OR network. This is illustrated in Table \ref{tab:AOnet} (third column).

\begin{table}
\begin{tabular}{ccc}
\hline
Fixed points & Structure from Lemma \ref{lemma:01} & ``Reduced'' system (Proposition \ref{prop:reduction}) \\
\hline
000000000000 & 000 0 0 000 00 00 & 00 0 0 00 00 00 \\
000000000011 & 000 0 0 000 00 11 & 00 0 0 00 00 11\\
000001110000 & 000 0 0 111 00 00 & 00 0 0 11 00 00\\
000001111111 & 000 0 0 111 11 11 & 00 0 0 11 11 11\\
000001110011 & 000 0 0 111 00 11 & 00 0 0 11 00 11 \\
000111110000 & 000 1 1 111 00 00 & 00 1 1 11 00 00 \\
000111110011 & 000 1 1 111 00 11 & 00 1 1 11 00 11 \\
000111111111 & 000 1 1 111 11 11 & 00 1 1 11 11 11 \\
111100000000 & 111 1 0 000 00 00 & 11 1 0 00 00 00 \\
111100000011 & 111 1 0 000 00 11 & 11 1 0 00 00 11 \\
111111110000 & 111 1 1 111 00 00 & 11 1 1 11 00 00 \\
111111110011 & 111 1 1 111 00 11 & 11 1 1 11 00 11 \\ 
111111111111 & 111 1 1 111 11 11 & 11 1 1 11 11 11 \\
\hline
\end{tabular}
\caption{Fixed points of the AND-OR network  $\wedge^2\!\vee\!\wedge\!\vee^3\!\wedge^2\vee$. First column: fixed points. Second column: fixed points with the structure given by Lemma \ref{lemma:01} highlighted. Third column: fixed points of reduced network, $\wedge^2\vee\wedge\vee^2\wedge^2\vee$, with the structure given by Lemma \ref{lemma:01}  highlighted.}
\label{tab:AOnet}
\end{table}

\begin{proposition}\label{prop:endpoint}
Let $r_1,\ldots,r_m$ in \{1,2\}, and $m\geq 2$. Then, we have the following  
\begin{displaymath}
\mathcal{F}(1,r_1,\ldots,r_m,1)  = 
\begin{cases}
\mathcal{F}(1,r_3,\ldots,r_{m},1) +
\mathcal{F}(r_3,\ldots,r_{m},1), 
& \text{for $r_1=1, r_2=1$}  \\
\mathcal{F}(2,r_3,\ldots,r_{m},1) +
\mathcal{F}(1,r_3,\ldots,r_{m},1), 
& \text{for $r_{1}=1, r_{2}=2$}  \\
\mathcal{F}(1,1,r_3,\ldots,r_{m},1) +
\mathcal{F}(r_3,\ldots,r_{m},1), 
& \text{for $r_{1}=2, r_{2}=1$}  \\
\mathcal{F}(1,2,r_3,\ldots,r_{m},1) +
\mathcal{F}(1,r_3,\ldots,r_{m},1), 
& \text{for $r_{1}=2, r_{2}=2$}  \\
\end{cases}
\end{displaymath}
\end{proposition}
\begin{proof}
We will use the notation of Lemma~\ref{lemma:01}.

If $r_1=1, r_2=1$, then we claim that any fixed point of $f=\wedge\vee\wedge\vee^{r_3}\wedge^{r_4}\cdots$ is of the form $\mathbf{x}=(\mathbf{x}^0,\mathbf{x}^1,\mathbf{x}^2,\ldots,\mathbf{x}^m,\mathbf{x}^{m+1})$ where either $\mathbf{x}^0=\mathbf{0}$ and $\mathbf{z}=(\mathbf{x}^1,\mathbf{x}^2,\ldots,\mathbf{x}^m,\mathbf{x}^{m+1})$ is a fixed point of $g=\wedge\vee^{r_3}\wedge^{r_4}\cdots$ or $\mathbf{x}^0=\mathbf{x}^1=\mathbf{1}$ and $\mathbf{z}=(\mathbf{x}^2,\ldots,\mathbf{x}^m,\mathbf{x}^{m+1})$ is a fixed point of $h=\vee^{r_3}\wedge^{r_4}\cdots$. Indeed, the system of Boolean equations for fixed points is 
\[
\begin{array}{lll}
x_1 & = & x_2 \\
x_2 & = & x_1\wedge x_3 \\
x_3 & = & x_2\vee x_4 \\
x_4 & = & x_3\wedge x_5 \\
x_5 & = & x_4\vee x_6 \\
    & \ \vdots & \\
x_n & = & x_{n-1}
\end{array}
\]
\begin{figure}[h]
\centerline{ \hbox{ \framebox{
\includegraphics[width=0.6\textwidth]{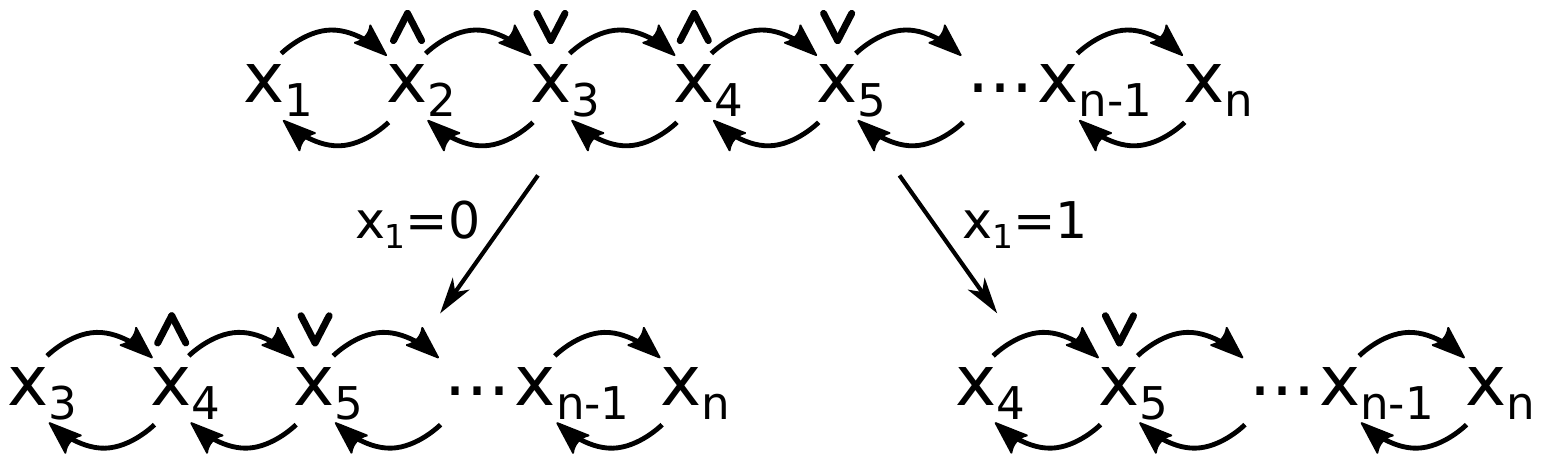}}}
 }\caption{Idea behind the proof of Proposition \ref{prop:endpoint} (logical operators are included for clarity). Considering the cases $x_1=0$ and $x_1=1$ yields systems of equations that correspond to smaller AND-OR networks.}
 \label{fig-chain_proof1}
\end{figure}

We divide this system of equations in the cases $x_1=0$ and $x_1=1$. Then, using the fact that $1=m\wedge n$ implies that $m=n=1$, that $0=m\vee n$ implies $m=n=0$, 
it follows that we obtain the two systems 

\[
\begin{array}{lll}
x_3 & = &  x_4 \\
x_4 & = & x_3\wedge x_5 \\
x_5 & = & x_4\vee x_6 \\
    & \ \vdots & \\
x_n & = & x_{n-1}
\end{array}
\]

and 

\[
\begin{array}{lll}
x_4 & = &  x_5 \\
x_5 & = & x_4\vee x_6 \\
    & \ \vdots & \\
x_n & = & x_{n-1},
\end{array}
\]
corresponding to the cases $x_1=0$ and $x_1=1$, respectively (see Fig.~\ref{fig-chain_proof1}). This means that the number of fixed points of $f$ is equal to the number of solutions of these two systems. Since the solutions of the first system are the fixed points of $g=\wedge\vee^{r_3}\wedge^{r_4}\cdots$ and the solutions of the second system are the fixed points of $h=\vee^{r_3}\wedge^{r_4}\cdots$, we obtain 
$\mathcal{F}(1,1,1,r_3\ldots,r_m,1)  = 
\mathcal{F}(1,r_3,\ldots,r_{m},1) +
\mathcal{F}(r_3,\ldots,r_{m},1)
$.

The proof for the other three cases is similar.
\end{proof}

By convention, we denote the AND-OR network $f(x_1,x_2)=(x_2,x_1)$ by an empty sequence, $()$. We also use the convention $\mathcal{F}(0,k_1,\ldots,k_m,0)=\mathcal{F}(k_1,\ldots,k_m,0)=\mathcal{F}(0,k_1,\ldots,k_m)=\mathcal{F}(k_1,\ldots,k_m)$ which will simplify the formulation of upcoming results.

\begin{theorem}\label{thm:main}
With the convention above, we have that for $m\geq 3$ and $k_i\geq 1$
\[
\mathcal{F}(k_1,\ldots,k_m) = \mathcal{F}(k_2-1,k_3,\ldots,k_m) + \mathcal{F}(k_3-1,k_4,\ldots,k_m)\]
and
\[
\mathcal{F}(k_1,\ldots,k_m) = \mathcal{F}(k_1,\ldots,k_{m-2},k_{m-1}-1) + \mathcal{F}(k_1,\ldots,k_{m-3},k_{m-2}-1).
\]
Also, \[ \mathcal{F}(k_1,k_2)=3, \ \  \mathcal{F}(k)=2 \text{ for $k\geq 0$}.\] 
\end{theorem}

\begin{proof}
For $m\geq 4$ the result follows directly from Propositions \ref{prop:reduction} and \ref{prop:endpoint}. For $m=3$ the results follows from $\mathcal{F}(1,2,1)=5$, $\mathcal{F}(1,1,1)=4$, $\mathcal{F}(1,1)=3$,  $\mathcal{F}(1)=2$, and $\mathcal{F}(0)=2$ which can be easily checked by complete enumeration.

\end{proof}

\textit{Example} \ref{eg:AOnet} (cont.) We now use Theorem \ref{thm:main} to find the number of fixed points of $\wedge^2\!\vee\!\wedge\!\vee^3\!\wedge^2\vee$:
\[
\begin{array}{lll}
\mathcal{F}(2,1,1,3,2,1) & = & \mathcal{F}(1,1,1,2,2,1) \\
& = &  \mathcal{F}(1-1,1,2,2,1) +\mathcal{F}(1-1,2,2,1) \\
& = &  \mathcal{F}(1,2,2,1) +\mathcal{F}(2,2,1) \\
& = &  \mathcal{F}(2-1,2,1) + \mathcal{F}(2-1,1) + 
       \mathcal{F}(2-1,1) + \mathcal{F}(1-1) \\
 & = &  \mathcal{F}(1,2,1) + \mathcal{F}(1,1) + 
       \mathcal{F}(1,1) + \mathcal{F}(0) \\
  & = &  \mathcal{F}(2-1,1) + \mathcal{F}(1-1)+
   \mathcal{F}(1,1) +         \mathcal{F}(1,1) + \mathcal{F}(0) \\      
  & = &  \mathcal{F}(1,1) + \mathcal{F}(0)+
   \mathcal{F}(1,1) +         \mathcal{F}(1,1) + \mathcal{F}(0) \\      
& = &  3 + 2 + 3 + 3 + 2 \\      
& = &  13 \\
\end{array}
\]
or
\[
\begin{array}{lll}
\mathcal{F}(2,1,1,3,2,1) & = & \mathcal{F}(1,1,1,2,2,1) \\
& = & \mathcal{F}(1,1,1,2,2-1) + \mathcal{F}(1,1,1,2-1) \\
& = & \mathcal{F}(1,1,1,2,1) + \mathcal{F}(1,1,1,1) \\
& = & \mathcal{F}(1,1,1,2-1) + \mathcal{F}(1,1,1-1) +
\mathcal{F}(1,1,1-1) + \mathcal{F}(1,1-1) \\
& = & \mathcal{F}(1,1,1,1) + \mathcal{F}(1,1) +
\mathcal{F}(1,1) + \mathcal{F}(1) \\
& = & \mathcal{F}(1,1,1-1) + \mathcal{F}(1,1-1) + 
\mathcal{F}(1,1) + \mathcal{F}(1,1) + \mathcal{F}(1) \\
& = & \mathcal{F}(1,1) + \mathcal{F}(1) + 
\mathcal{F}(1,1) + \mathcal{F}(1,1) + \mathcal{F}(1) \\
& = & 3 +2 +3 +3 +2\\
& = & 13
\end{array}
\]

In this way, Theorem \ref{thm:main} provides a recursive formula to compute the number of fixed points of AND-OR networks with chain topology without the need of exhaustive enumeration. We now study 2 especial cases $\mathcal{F}(1,1,\ldots,1,1)$ and $\mathcal{F}(2,2,\ldots,2,2)$.

Define $A_n=(1,\underbrace{1,1,\ldots,1,1}_{n \text{ times}},1)$ and $B_n=(2,\underbrace{2,2,\ldots,2,2}_{n \text{ times}},2)$. Define the sequences $a_0=1$, $a_1=1$, $a_2=1$, and $a_{n}=a_{n-2}+a_{n-3}$ for $n\geq 3$ and $b_0=1$, $b_1=1$, and $b_{n}=b_{n-1}+b_{n-2}$ for $n\geq 2$. Note that $(a_n)$ is the  Padovan sequence and $(b_n)$ is the Fibonacci sequence.

\begin{corollary}\label{cor:anbn}
 With the definitions above we have $\mathcal{F}(A_n)=a_{n+5}$ and $\mathcal{F}(B_n)=b_{n+3}$ for $n\geq 0$, and the sharp bounds $\mathcal{F}(A_n)\leq \mathcal{F}(1,r_1,r_2,\ldots,r_n,1) \leq\mathcal{F}(B_n)$ for all $r_i\geq 1$.
\end{corollary}
\begin{proof}
It follows from Theorem \ref{thm:main} or Proposition \ref{prop:endpoint} using induction.
\end{proof}

\section{Infinite and Closed Chain}
\label{sec:infchain}

In this section we study the cases of AND-OR networks with infinitely many variables and when the topology is a closed chain. 

When the AND-OR network has infinitely many variables we have a collection of Boolean functions $f=(\ldots,f_{-2},f_{-1},f_0,f_1,f_2,\ldots)$ such that $f_i=x_{i-1}\wedge x_{i+1}$ or $f_i=x_{i-1}\vee x_{i+1}$. We can use the same notation of Section \ref{sec:openchain} and denote consecutive logical operators as $\wedge^k$ or $\vee^k$, where $k$ could also be $\infty$. Also, we can simply use the exponents to represent the AND-OR network. For example, $(\infty,1,2,\infty)$ and $\wedge^\infty \vee \wedge^2\vee^\infty$ represent the AND-OR network $\ldots\wedge\wedge\wedge\vee\wedge\wedge\vee\vee\vee\ldots$, and $(\ldots,1,1,2,1,1,2,1,1,2,\ldots)$ and $\ldots\wedge\vee\wedge^2\vee\wedge\vee^2\wedge\vee\wedge^2 \ldots$ represent the AND-OR network $\ldots\wedge\vee\wedge\wedge\vee\wedge\vee\vee\wedge\vee\wedge\wedge \ldots$. 

The following theorem allows us to use the results from Section \ref{sec:openchain} to study AND-OR networks with infinitely many variables. 

\begin{theorem}\label{thm:inf}
With the notation above and $k_i\geq 1$ we have the following.
\[
\begin{array}{rll}
\mathcal{F}(\infty) & = &  2 \\
\mathcal{F}(\infty,k_1,k_2,\ldots,k_{m-1},k_m,\infty) & = &  \mathcal{F}(1,k_1,k_2,\ldots,k_{m-1},k_m,1) \\
\mathcal{F}(\infty,k_1,k_2,k_3,\ldots) & = &  \infty \\
\mathcal{F}(\ldots,k_{-3},k_{-2},k_{-1},\infty) & = &  \infty \\
\mathcal{F}(\ldots,k_{-3},k_{-2},k_{-1},k_0,k_1,k_2,k_3,\ldots) & = &  \infty \\
\end{array}
\]
\end{theorem} 
\begin{proof}
To prove the first equality we consider the AND-OR network where all logical operators are $\wedge$. If one of the variables is 0, it follows that all the other variables are also 0. Similarly, if one of the variables is 1, all the other variables are also 1. Thus, the only fixed points of this AND-OR network are $\mathbf{0}$ and $\mathbf{1}$.

The second equality follows the same approach seen in Proposition \ref{prop:reduction}.

To prove the third equality we first observe that $\mathcal{F}(\infty,k_1,k_2,k_3,\ldots)=\mathcal{F}(1,k_1,k_2,k_3,\ldots)$. Now, we will show that any fixed point of the AND-OR network $\mathcal{F}(1,k_1,k_2,k_3,\ldots,k_r)$ defines a fixed point of $\mathcal{F}(1,k_1,k_2,k_3,\ldots)$. Indeed, using the notation of Lemma \ref{lemma:01}, a fixed point of the AND-OR network $\mathcal{F}(1,k_1,\ldots,k_r)$ has the form $\mathbf{x}=(\mathbf{x}^0,\mathbf{x}^1,\ldots,\mathbf{x^r})$. Then, denoting $\mathbf{z}=(1,1,\ldots)$ if $\mathbf{x}^r=\mathbf{1}$ and  $\mathbf{z}=(0,0,\ldots)$ if $\mathbf{x}^r=\mathbf{0}$, it follows that $(\mathbf{x}^0,\mathbf{x}^1,\ldots,\mathbf{x^r},\mathbf{z})$ is a fixed point of $\mathcal{F}(1,k_1,k_2,k_3,\ldots)$. Since $r$ is arbitrary, $\mathcal{F}(1,k_1,\ldots,k_r)$ is not bounded (see Corollary \ref{cor:anbn}) and then number of fixed points of $\mathcal{F}(1,k_1,\ldots)$ is $\infty$. The last two equalities are similar.
\end{proof}

When the topology of the network is a closed chain, we have the network
\begin{displaymath}
\begin{array}{lllllll}
f_1=x_n\diamondsuit_1 x_2, & 
f_2=x_1\diamondsuit_2 x_3, &
f_3=x_2\diamondsuit_3 x_4, &
\ldots & , &
f_{n-1}=x_{n-2}\diamondsuit_{n-1} x_n, &
f_n=x_{n-1}\diamondsuit_n x_1.
\end{array}
\end{displaymath}
We denote this network as $[k_1,k_2,\ldots,k_r]$ or any cyclic permutation that groups consecutive logical operators. Thus, the AND-OR network 
\begin{displaymath}
\begin{array}{lllllll}
f_1=x_n \wedge x_2, & 
f_2=x_1 \vee x_3, &
f_3=x_2 \wedge x_4, &
f_4=x_3 \vee x_5, &
f_5=x_4 \vee x_6, &
f_6=x_5 \wedge x_1, 
\end{array}
\end{displaymath}
will not be denoted by $[1,1,1,2,1]$ (``splitting'' the first and last $\wedge$'s), but by $[1,1,2,2]$, $[1,2,2,1]$, $[2,2,1,1]$, or $[2,1,1,2]$ (combining the first and last $\wedge$'s). This means that $r$ in $[k_1,k_2,\ldots,k_r]$ will always be an even number or equal to 1. The number of fixed points will be denoted by $\mathcal{F}[k_1,k_2,\ldots,k_r]$.
The following propositions and theorem allow us to use the results from Section \ref{sec:openchain} to study AND-OR networks with closed chain topology. 

\begin{proposition}\label{prop:reduction_closed}
With the notation above, we have that for $k_i\geq 1$
\[
\mathcal{F}[k_1,k_2,\ldots,k_r] = \mathcal{F}[\min\{2,k_1\},\min\{2,k_2\},\ldots,\min\{2,k_r\}]. 
\]
\end{proposition}
\begin{proof}
It is analogous to the proof of Proposition \ref{prop:reduction}.
\end{proof}

\begin{proposition}\label{prop:closed}
Consider $k_i\geq 1$, $m\geq 6$, and $l\geq 8$. Then,
\[
\begin{array}{rll}
\mathcal{F}[2,k_2,\ldots,k_m] & = & 
\mathcal{F}(k_2-1,k_3,\ldots,k_{m-1},k_m-1)+
\mathcal{F}(k_3-1,k_4,\ldots,k_{m-2},k_{m-1}-1),   \\
\mathcal{F}[1,k_2,\ldots,k_l] & = &  
\mathcal{F}(k_3-1,k_4,\ldots,k_{l-1}-1) +
\mathcal{F}(k_4-1,k_5,\ldots,k_{l-1},k_l-1) \ + \\
& & \mathcal{F}(k_2-1,k_3,\ldots,k_{l-3},k_{l-2}-1) -
\mathcal{F}(k_4-1,k_5,\ldots,k_{l-3},k_{l-2}-1).
\end{array}
\]
\end{proposition}
\begin{proof}
The first equality is analogous to Proposition \ref{prop:endpoint}. To prove the second equality we use the notation of Lemma \ref{lemma:01}. 

We have several cases to consider for $k_{l-2}$, $k_{l-1}$, $k_l$, $k_2$, $k_3$, and $k_4$. We focus on the case $k_{l-2}=k_{l-1}=k_l=k_2=k_3=k_4=1$ since the other cases are analogous. Note that we want to prove 
\[
\begin{array}{ll}
\mathcal{F}[1,1,1,1,k_5\ldots,k_{l-3},1,1,1]  = &  
\mathcal{F}(1,k_5,\ldots,k_{l-3},1) +
\mathcal{F}(k_5,\ldots,k_{l-3},1,1) \ +  \\
 & \mathcal{F}(1,1,k_5,\ldots,k_{l-3}) -
\mathcal{F}(k_5,\ldots,k_{l-3}).
\end{array}
\]

The fixed points of the AND-OR network are the solutions of
\[
\begin{array}{rll}
x_1 & = & x_{n} \wedge x_2,  \\
x_2 & = & x_1 \vee x_3, \\
x_3 & = & x_2 \wedge x_4, \\
x_4 & = & x_3 \vee x_5, \\
x_5 & = & x_4 \wedge x_6, \\
 & \ \vdots \\
x_{n-3} & = & x_{n-4}\wedge x_{n-2}, \\ 
x_{n-2} & = & x_{n-3}\vee x_{n-1}, \\ 
x_{n-1} & = & x_{n-2}\wedge x_n, \\ 
x_{n} & = & x_{n-1}\vee x_1.
\end{array}
\]

\begin{figure}[h]
\centerline{ \hbox{ \framebox{
\includegraphics[width=0.97\textwidth]{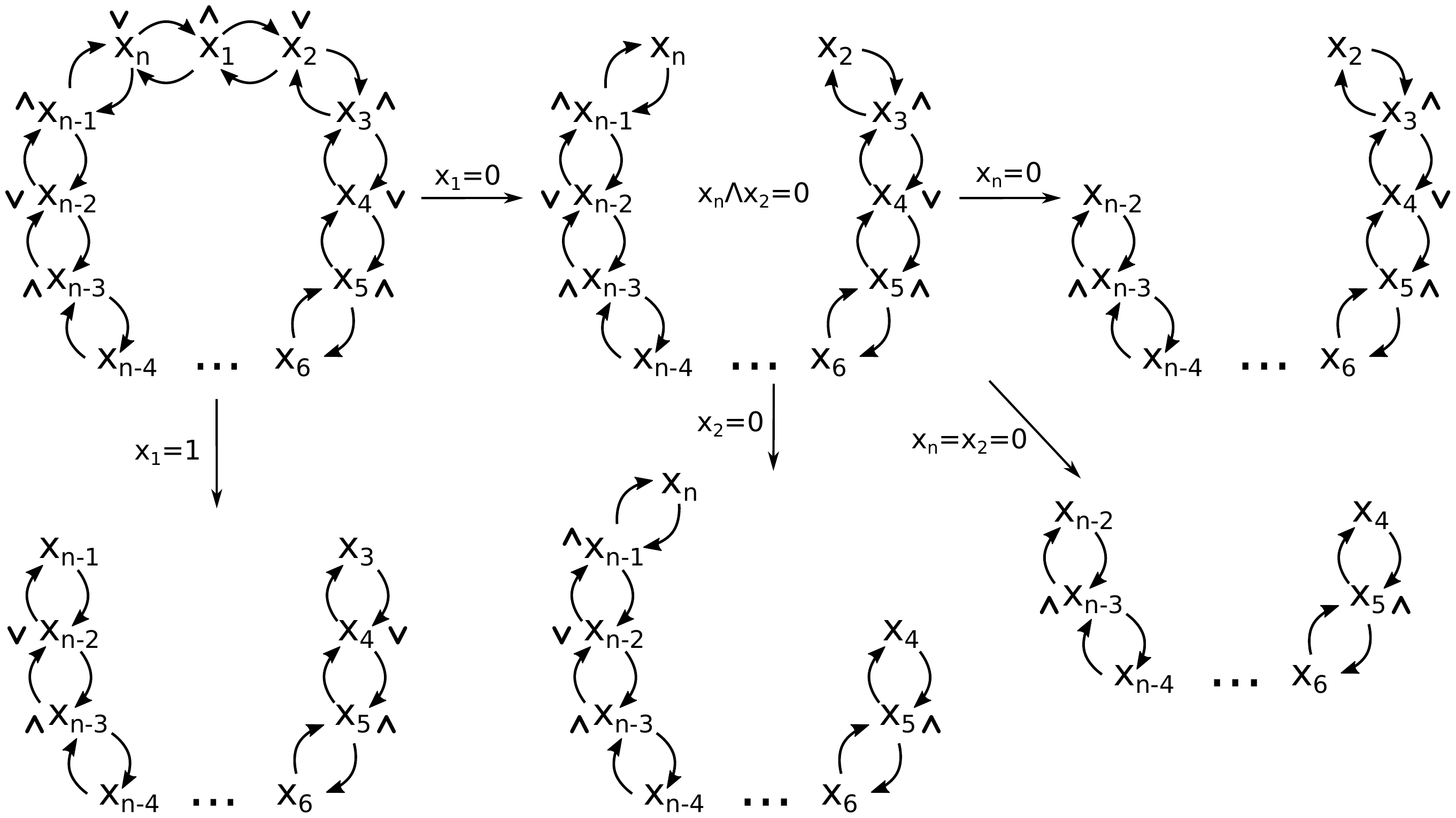}}}
 }\caption{Idea behind the proof of Proposition \ref{prop:closed}  (logical operators are included for clarity). Considering the case $x_1=1$ yields a system of equations that corresponds to a smaller AND-OR network. Considering the case $x_1=0$ yields a system of equation that does not correspond to an AND-OR network (due to the equation $x_n\wedge x_2=0$). However, the subcases $x_n=0$ and $x_2=0$ yield systems of equations that do correspond to smaller AND-OR networks. These two systems have overlapping solutions, so we must also take into consideration the common case $x_n=x_2=0$ when counting the number of fixed points.}
 \label{fig-chain_proof2}
\end{figure}

We now consider the cases $x_1=1$ and $x_1=0$ (see Fig.\ref{fig-chain_proof2}). The case $x_1=1$ yields the system of equations 
\[
\begin{array}{rll}
x_3 & = & x_4, \\
x_4 & = & x_3 \vee x_5, \\
x_5 & = & x_4 \wedge x_6, \\
 & \ \vdots \\
x_{n-3} & = & x_{n-4}\wedge x_{n-2}, \\ 
x_{n-2} & = & x_{n-3}\vee x_{n-1}, \\ 
x_{n-1} & = & x_{n-2},
\end{array}
\]
which has $\mathcal{F}(1,k_5,\ldots,k_{l-3},1)$ solutions. On the other hand, when we consider $x_1=0$ the first equation becomes $x_n\wedge x_2=0$. We now have 2 subcases: $x_n=0$ and $x_2=0$. The subcase $x_n=0$ yields
\[
\begin{array}{rll}
x_2 & = & x_3, \\
x_3 & = & x_2 \wedge x_4, \\
x_4 & = & x_3 \vee x_5, \\
x_5 & = & x_4 \wedge x_6, \\
 & \ \vdots \\
x_{n-3} & = & x_{n-4}\wedge x_{n-2}, \\ 
x_{n-2} & = & x_{n-3},
\end{array}
\]
which has $\mathcal{F}(1,1,k_5,\ldots,k_{l-3})$ solutions. The subcase $x_2=0$ yields 
\[
\begin{array}{rll}
x_4 & = & x_5, \\
x_5 & = & x_4 \wedge x_6, \\
 & \ \vdots \\
x_{n-3} & = & x_{n-4}\wedge x_{n-2}, \\ 
x_{n-2} & = & x_{n-3}\vee x_{n-1}, \\ 
x_{n-1} & = & x_{n-2}\wedge x_n, \\ 
x_{n} & = & x_{n-1},
\end{array}
\]
which has $\mathcal{F}(k_5,\ldots,k_{l-3},1,1)$ solutions. Thus, adding up these 3 numbers we obtain  $\mathcal{F}(1,k_5,\ldots,k_{l-3},1) +
\mathcal{F}(k_5,\ldots,k_{l-3},1,1) + \mathcal{F}(1,1,k_5,\ldots,k_{l-3})$. 
However, this is not $\mathcal{F}[1,1,1,1,k_5\ldots,k_{l-3},1,1,1]$, since the subcases $x_n=0$ and $x_2=0$ overlap. We need to subtract the number of solutions of the system
\[
\begin{array}{rll}
x_4 & = &  x_5, \\
x_5 & = & x_4 \wedge x_6, \\
 & \ \vdots \\
x_{n-3} & = & x_{n-4}\wedge x_{n-2}, \\ 
x_{n-2} & = & x_{n-3},
\end{array}
\]
which has $\mathcal{F}(k_5,\ldots,k_{l-3})$ solutions. Then, the result follows.
\end{proof}

We now declare some conventions to write Proposition \ref{prop:closed} more compactly. We define $\mathcal{F}(-1)=1$, $(k_s-1,\ldots,k_s-1)=(k_s-2)$, and  $(k_s-1,\ldots,k_t-1)=(-1)$ for $s>t$.

\begin{theorem}\label{thm:closed}
With the conventions above, we have that for $m\geq 4$ and $k_i\geq 1$
\[
\begin{array}{rll}
\mathcal{F}[2,k_2,\ldots,k_r] & = & 
\mathcal{F}(k_2-1,k_3,\ldots,k_{r-1},k_r-1)+
\mathcal{F}(k_3-1,k_4,\ldots,k_{r-2},k_{r-1}-1),   \\
\mathcal{F}[1,k_2,\ldots,k_r] & = &  
\mathcal{F}(k_3-1,k_4,\ldots,k_{r-1}-1) +
\mathcal{F}(k_4-1,k_5,\ldots,k_{r-1},k_r-1) \ + \\
& & \mathcal{F}(k_2-1,k_3,\ldots,k_{r-3},k_{r-2}-1) -
\mathcal{F}(k_4-1,k_5,\ldots,k_{r-3},k_{r-2}-1). \\
\end{array}
\]

Also,
\[
\begin{array}{rll}
\mathcal{F}[k] & = &  2 \text{ \ \ for $k\geq 3$}, \\
\mathcal{F}[k,1] & = &  2  \text{ \ \ for $k\geq 2$}, \\
\mathcal{F}[k_1,k_2] & = & 3 \text{ \ \ for $k_1,k_2\geq 2$}, 
\end{array}
\]
\end{theorem}
\begin{proof}
The first two equalities follows directly from Proposition \ref{prop:reduction_closed} and \ref{prop:closed} using the convention declared above. The last 3 equalities follow from Proposition \ref{prop:reduction_closed} and $\mathcal{F}[3]=\mathcal{F}[2,1]=2$ and $\mathcal{F}[2,2]=3$, which can be verified by complete enumeration.
\end{proof}

As in Section \ref{sec:openchain}, we now consider the cases  $A_n=(1,\underbrace{1,1,\ldots,1,1}_{n \text{ times}},1)$ and $B_n=(2,\underbrace{2,2,\ldots,2,2}_{n \text{ times}},2)$. We denote the number of fixed points of the corresponding AND-OR networks with closed chain topology as $\mathcal{F}[A_n]$ and $\mathcal{F}[B_n]$, respectively.

\begin{corollary} With the notation above we have $\mathcal{F}[A_n]=3a_{n}-a_{n-2}$ and  $\mathcal{F}[B_n]=b_{n+2}+b_{n}$ for $n\geq 2$, and the sharp bounds $\mathcal{F}[A_{n}] \leq \mathcal{F}[k_0,k_1,\ldots,k_n,k_{n+1}]\leq \mathcal{F}[B_{n}]$ for all $r_i\geq 1$
\end{corollary}
\begin{proof}
The proof follows from Theorem \ref{thm:closed} and Corollary \ref{cor:anbn}.
\end{proof}

\begin{example}\label{eg:AOnet_closed}
We consider 
\begin{displaymath}
\begin{array}{llllll}
f_1=x_{12}\wedge x_2,& 
f_2=x_1\wedge x_3, &
f_3=x_2\wedge x_4, &
f_4=x_3\vee x_5, &
f_5=x_4\wedge x_6, &
f_6=x_5\vee x_7,\\
f_7=x_6\vee x_8, &
f_8=x_7\vee x_9, &
f_9=x_8\wedge x_{10}, &
f_{10}=x_9\wedge x_{11}, &
f_{11}=x_{10}\vee x_{12}, &
f_{12}=x_{11}\vee x_1.
\end{array}
\end{displaymath}
We will use Theorems \ref{thm:main} and  \ref{thm:closed} for the representations $[3,1,1,3,2,2]$ and $[1,3,2,2,3,1]$ of $f$.
\[
\begin{array}{lll}
\mathcal{F}[3,1,1,3,2,2] & = & \mathcal{F}[2,1,1,2,2,2] \\
& = &  \mathcal{F}(1-1,1,2,2,2-1) +\mathcal{F}(1-1,2,2-1) \\
& = &  \mathcal{F}(1,2,2,1) +\mathcal{F}(2,1) \\
& = &  \mathcal{F}(2-1,2,1)+\mathcal{F}(2-1,1) +\mathcal{F}(2,1) \\
& = &  \mathcal{F}(1,2,1)+\mathcal{F}(1,1) +\mathcal{F}(2,1) \\
& = &  \mathcal{F}(2-1,1)+\mathcal{F}(1-1)+\mathcal{F}(1,1) +\mathcal{F}(2,1) \\
& = & \mathcal{F}(1,1)+\mathcal{F}(0)+\mathcal{F}(1,1)+\mathcal{F}(2,1)\\
& = & 3+2+3 +3=11 
\end{array}
\]

\[
\begin{array}{lll}
\mathcal{F}[1,3,2,2,3,1] & = & \mathcal{F}[1,2,2,2,2,1] \\
 & = & \mathcal{F}(2-1,2,2-1) + \mathcal{F}(2-1,2,1-1) +
\mathcal{F}(2-1,2,2-1) - \mathcal{F}(2-2)\\
& = & \mathcal{F}(1,2,1) + \mathcal{F}(1,2) +
\mathcal{F}(1,2,1) - \mathcal{F}(0)\\
& = & \mathcal{F}(1,1)+\mathcal{F}(0) + \mathcal{F}(1,2) +
\mathcal{F}(1,1)+\mathcal{F}(0) - \mathcal{F}(0)\\
& = & 3+2 + 3 +3+2 - 2=11\\

\end{array}
\]

\end{example}

\section{Conclusion}

Our results provide recursive formulas and sharp bounds for the number of fixed points of AND-OR networks with chain topology. Other work regarding the number of fixed points has focused on bounds with respect to the number of nodes \citep{ADG}. Our results, on the other hand, focus on formulas and bounds with respect to the number of consecutive logical operators. Thus, our results complement previous results.

Our approach can potentially be extended to cases where an AND-OR network has a topology that can be seen as the ``combination'' of open chains. Then, the number of fixed points of the original AND-OR network will be given by the inclusion-exclusion principle in terms of the number of fixed points of the AND-OR networks with open chain topology. Indeed, Theorem \ref{thm:closed} shows how our approach can be used in such cases. 

\bibliographystyle{abbrvnat}
\bibliography{ref}

\begin{thebibliography}{14}
\providecommand{\natexlab}[1]{#1}
\providecommand{\url}[1]{\texttt{#1}}
\expandafter\ifx\csname urlstyle\endcsname\relax
  \providecommand{\doi}[1]{doi: #1}\else
  \providecommand{\doi}{doi: \begingroup \urlstyle{rm}\Url}\fi

\bibitem[Agur et~al.(1988)Agur, Fraenkel, and Klein]{AFK}
Z.~Agur, A.~Fraenkel, and S.~Klein.
\newblock The number of fixed points of the majority rule.
\newblock \emph{Discrete Math.}, 70\penalty0 (3):\penalty0 295--302, 1988.
\newblock ISSN 0012-365X.

\bibitem[Akutsu et~al.(1998)Akutsu, Kuhara, Maruyama, and Miyano]{Akutsu}
T.~Akutsu, S.~Kuhara, O.~Maruyama, and S.~Miyano.
\newblock A system for identifying genetic networks from gene expression
  patterns produced by gene disruptions and overexpressions.
\newblock \emph{Genome Inform.}, 9:\penalty0 151--160, 1998.

\bibitem[Albert and Othmer(2003)]{AO}
R.~Albert and H.~Othmer.
\newblock The topology of the regulatory interactions predicts the expression
  pattern of the segment polarity genes in \emph{Drosophila melanogaster}.
\newblock \emph{J. Theor. Biol.}, 223:\penalty0 1--18, 2003.

\bibitem[Aracena(2008)]{Aracena_FB_BRN}
J.~Aracena.
\newblock Maximum number of fixed points in regulatory boolean networks.
\newblock \emph{Bulletin of Mathematical Biology}, 70\penalty0 (5):\penalty0
  1398--1409, 2008.
\newblock URL \url{DOI:10.1007/s11538-008-9304-7}.

\bibitem[Aracena et~al.(2004)Aracena, Demongeot, and Goles]{ADG}
J.~Aracena, J.~Demongeot, and E.~Goles.
\newblock Fixed points and maximal independent sets in {AND}-{OR} networks.
\newblock \emph{Discrete Appl. Math.}, 138\penalty0 (3):\penalty0 277--288,
  2004.
\newblock ISSN 0166-218X.

\bibitem[Jarrah et~al.(2007)Jarrah, Raposa, and Laubenbacher]{Jarrah2007167}
A.~Jarrah, B.~Raposa, and R.~Laubenbacher.
\newblock Nested canalyzing, unate cascade, and polynomial functions.
\newblock \emph{Physica D: Nonlinear Phenomena}, 233\penalty0 (2):\penalty0 167
  -- 174, 2007.
\newblock ISSN 0167-2789.
\newblock \doi{10.1016/j.physd.2007.06.022}.
\newblock URL
  \url{http://www.sciencedirect.com/science/article/pii/S0167278907002035}.

\bibitem[Jarrah et~al.(2010)Jarrah, Laubenbacher, and Veliz-Cuba]{CBN}
A.~Jarrah, R.~Laubenbacher, and A.~Veliz-Cuba.
\newblock The dynamics of conjunctive and disjunctive {B}oolean network models.
\newblock \emph{Bull. Math. Bio.}, 72\penalty0 (6):\penalty0 1425--1447, 2010.
\newblock \doi{10.1007/s11538-010-9501-z}.

\bibitem[Mendoza and Xenarios(2006)]{mendozamethod}
L.~Mendoza and I.~Xenarios.
\newblock A method for the generation of standardized qualitative dynamical
  systems of regulatory networks.
\newblock \emph{Theoretical Biology and Medical Modelling}, 3\penalty0
  (1):\penalty0 13, 2006.
\newblock ISSN 1742-4682.
\newblock \doi{10.1186/1742-4682-3-13}.
\newblock URL \url{http://www.tbiomed.com/content/3/1/13}.

\bibitem[Murrugarra and Laubenbacher(2011)]{Murrugarra201166}
D.~Murrugarra and R.~Laubenbacher.
\newblock Regulatory patterns in molecular interaction networks.
\newblock \emph{Journal of Theoretical Biology}, 288\penalty0 (0):\penalty0 66
  -- 72, 2011.
\newblock ISSN 0022-5193.
\newblock \doi{10.1016/j.jtbi.2011.08.015}.
\newblock URL
  \url{http://www.sciencedirect.com/science/article/pii/S0022519311004103}.

\bibitem[Veliz-Cuba and Laubenbacher(2011)]{SCBN}
A.~Veliz-Cuba and R.~Laubenbacher.
\newblock On the computation of fixed points in {B}oolean networks.
\newblock \emph{Journal of Applied Mathematics and Computing}, 39\penalty0
  (1-2):\penalty0 145--153, 2011.

\bibitem[Veliz-Cuba and Stigler(2011)]{Velizlacop}
A.~Veliz-Cuba and B.~Stigler.
\newblock Boolean models can explain bistability in the \textit{lac} operon.
\newblock \emph{J. Comput. Biol.}, 18\penalty0 (6):\penalty0 783--794, 2011.

\bibitem[Veliz-Cuba et~al.(2013)Veliz-Cuba, Buschur, Hamershock, Kniss, Wolff,
  and Laubenbacher]{bnandnot}
A.~Veliz-Cuba, K.~Buschur, R.~Hamershock, A.~Kniss, E.~Wolff, and
  R.~Laubenbacher.
\newblock {AND-NOT} logic framework for steady state analysis of {B}oolean
  network models.
\newblock \emph{Applied Mathematics and Information Sciences}, 7\penalty0
  (4):\penalty0 1263--1274, 2013.

\bibitem[Veliz-Cuba et~al.(2014)Veliz-Cuba, Kumar, and
  Josi{\'c}]{veliz2014piecewise}
A.~Veliz-Cuba, A.~Kumar, and K.~Josi{\'c}.
\newblock Piecewise linear and {B}oolean models of chemical reaction networks.
\newblock \emph{Bulletin of Mathematical Biology}, 76\penalty0 (12):\penalty0
  2945--2984, 2014.

\bibitem[Veliz-Cuba et~al.(2015)Veliz-Cuba, Aguilar, and
  Laubenbacher]{veliz2015dimension}
A.~Veliz-Cuba, B.~Aguilar, and R.~Laubenbacher.
\newblock Dimension reduction of large sparse {AND-NOT} network models.
\newblock \emph{Electronic Notes in Theoretical Computer Science},
  316:\penalty0 83--95, 2015.

\end{thebibliography}
\label{sec:biblio}

\end{document}